\newcommand{\abs}[1]{\left\vert#1\right\vert}
\documentclass{article}

\usepackage[top=1.5in, left=1in, right=1in, bottom=1in]{geometry}

\usepackage[utf8]{inputenc}
\usepackage{gensymb}
\usepackage{amsmath,amsthm,verbatim,amssymb,amsfonts,amscd}
\usepackage{graphicx}
\usepackage{systeme}
\usepackage{xcolor}
\usepackage{hyperref}
\usepackage{caption}
\usepackage{subcaption}
\usepackage{scrextend}
\usepackage{lineno}
\usepackage[symbol]{footmisc}

\newtheorem{theorem}{Theorem}[section]
\newtheorem{lemma}[theorem]{Lemma}
\newtheorem{remark}[theorem]{Remark}
\newtheorem{case}{Case}
\newtheorem{observation}[theorem]{Observation}
\newtheorem{subcase}{Sub-Case}[case]
\theoremstyle{definition}
\newtheorem{definition}[theorem]{Definition}

\DeclareMathOperator{\p}{P}
\DeclareMathOperator{\Z}{Z}
\DeclareMathOperator{\pt}{pt}
\DeclareMathOperator{\ct}{ct}
\DeclareMathOperator{\Term}{Term}

\begin{document}

	\title{The zero forcing numbers and propagation times of gear graphs and helm graphs}
	
\author{Sara Anderton\thanks{Hartford, CT 06040 (sarajanderton@gmail.com)}\and Rilee Burden\thanks{Department of Mathematics, University of North Texas, Denton, TX 76205 (rileeburden@my.unt.edu).} \and McKenzie Fontenot \thanks{Department of Mathematics, University of North Texas, Denton, TX 76205 (kenzie.fontenot@unt.edu).} \and Noah Fredrickson \thanks{Flower Mound, TX 75022 (noahkfredrickson@gmail.com)} \and Alexandria Kwon \thanks{Little Elm, TX 75068 (kwonlexie@gmail.com)} \and Sydney Le \thanks{Department of Electrical and Computer Engineering, Rice University, Houston, TX 77005 (sydney.le@rice.edu)} \and Kanno Mizozoe \thanks{Department of Mathematics, Trinity College, Hartford, CT 06106 (kanno.mizozoe@trincoll.edu).} \and Erin Raign \thanks{Department of Mathematics, University of North Texas, Denton, TX 76205
(erinraign@my.unt.edu).} \and August Sangalli \thanks{Department of Mathematics, University of Denver, Denver, CO 80210 (gus.sangalli@du.edu).} \and Houston Schuerger \thanks{Department of Mathematics, Trinity College, Hartford, CT 06106 (houston.schuerger@trincoll.edu).} \and Andrew Schwartz\thanks{Department of Mathematics, Southeast Missouri State University, Cape Girardeau, MO 63701 (aschwartz@semo.edu)}}

\maketitle

\begin{abstract}
Zero forcing is a dynamic coloring process on graphs. Initially, each vertex of a graph is assigned a color of either blue or white, and then a process begins by which blue vertices force white vertices to become blue.  The zero forcing number is the cardinality of the smallest set of initially blue vertices which can force the entire graph to become blue, and the propagation time is the minimum number of steps in such a zero forcing process.  In this paper we will determine the zero forcing numbers and propagation times of two infinite classes of graphs called gear graphs and helm graphs.
\end{abstract}

\noindent {\bf Keywords:} graph classes, zero forcing, propagation time, relaxed chronology, chain set.

\noindent {\bf AMS Subject Classification:} 05C50, 05C69, 05C57

\section{Introduction}

Graph theory is classified as a branch of discrete mathematics. It studies mathematical objects called graphs, which use dots (called vertices) to represent locations or elements of a set and edges between them to visualize the relationships between the objects. Due to the adaptability of the structure, graph theory offers tools for modeling and analyzing pairwise relationships between objects in both mathematics and applications outside of mathematics. 
In 2008, a graph theory concept, zero forcing, was introduced by the AIM Minimum Rank - Special Graphs Work Group \cite{AIM}.

Zero forcing is a graph coloring game where the goal is to color all vertices blue using the fewest initial blue vertices on a graph where each vertex is initially colored blue or white. The color change rule allows a blue vertex to force its only white neighbor to become blue, and the game ends when no more white vertices can be colored.
This mathematical technique for analyzing graphs is used to study a variety of graph parameters, such as the maximum nullity and minimum rank of graphs. Besides graph theory, it was independently introduced in two fields: in physics as quantum control theory \cite{graphinfect}, and in the monitoring of power grids as power domination \cite{powdom} (with the role of zero forcing evident in \cite{powdomzf}).

The propagation time of a zero forcing set, which describes the number of steps needed to color all vertices of a graph blue, is one of the indicators to evaluate the performance of zero forcing. This concept is relatively new, having been proposed only eleven years ago \cite{proptime}.

This paper focuses on two infinite classes of graphs, gear graphs and helm graphs.  While the definitions we use for each class of graphs is slightly more general than those standardly used, in both cases the graph class we consider contains the more traditional version of the graph class as a subclass.  In section 2 we determine the zero forcing number and propagation time of gear graphs, and in section 3 we determine the zero forcing number and propagation time of helm graphs.  However, first we will need a collection of general graph theoretical definitions as well as certain concepts central to the topics of zero forcing and propagation time.  

A graph, $G$, consists of a set of {\em vertices} $V(G)$ and a set $E(G)$ of pairs of vertices called {\em edges}.  To distinguish between sets of vertices of size $2$ and edges, given two vertices $u,v \in V(G)$, the set containing only the vertices $u$ and $v$ will be denoted $\{u,v\}$ and an edge between $u$ and $v$ will be denoted $uv$.  If $uv \in E(G)$, then we say that $u$ and $v$ are {\em adjacent} (or {\em neighbors}) and that the edge $uv$ is {\em incident} to the vertices $u$ and $v$. Given a vertex $v \in V(G)$, the neighborhood of $v$, denoted $N_G(v)$, is the set of vertices $N_G(v)=\{u:uv \in E(G)\}$ and the degree of $v$, denoted $\deg(v)$, is the number $\deg(v)=\abs{N_G(v)}$. The minimum degree of $G$, denoted $\delta(G)$, is the smallest degree among all vertices in $G$.  If $\deg(v)=1$, then $v$ is said to be a {\em pendant vertex}.  If $G$ and $H$ are graphs such that $V(H)\subseteq V(G)$ and $E(H) \subseteq E(G)$, then $H$ is said to be a {\em subgraph of $G$}. If, in addition, for each pair of vertices $u,v \in V(H)$ we have that $uv \in E(H)$ if and only if $uv \in E(G)$, then $H$ is said to be an {\em induced subgraph} of $G$.  

A {\em path} is a sequence of vertices $(v_0,v_1,\dots,v_k)$ such that for each $i$ with $0 \leq i \leq k-1$ we have that $v_iv_{i+1} \in E(G)$.  Similarly, a {\em cycle} is a sequence of vertices $(v_0,v_1,\dots,v_k,v_0)$ such that $v_0v_k \in E(G)$ and for each $i$ with $0 \leq i \leq k-1$ we have that $v_iv_{i+1} \in E(G)$.  Paths can also be viewed as graphs in and of themselves.  The {\em path graph} on $n$ vertices, denoted $P_n$, is the graph with vertex set $V(P_n)=\{v_i\}_{i=1}^n$ and edge set such that for the distinct pair $i,j \in \{1,2,\dots,n\}$, $v_iv_j \in E(P_n)$ if and only if $\abs{i-j}=1$.  A {\em path cover} $\mathcal Q$ of a graph $G$ is a collection of vertex-induced path subgraphs $Q_1,Q_2,\dots,Q_n$ such that $\{V(Q_i)\}_{i=1}^n$ is a partition of $V(G)$. The {\em path cover number} of $G$, denoted $\p(G)$, is the minimum cardinality of a path cover of $G$.  A path cover $\mathcal Q$ of a graph $G$ is said to be {\em minimum} if $\abs{\mathcal Q}=\p(G)$.

\begin{figure}[h]
  \centering
  \includegraphics[width=.35\linewidth]{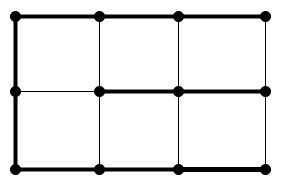}
  \caption{Example of a minimum path cover with path edges in bold}
  \label{path_cover_not_a_chain_set}
\end{figure}

Zero forcing is a dynamic coloring process on graphs.  Let $G$ be a graph. 
 Prior to beginning the process, the vertices of $G$ are all colored either blue or white.  Once this initial coloring is chosen, the zero forcing process begins.  To understand the process, one must first consider the zero forcing color change rule.  

\vspace{0.1in}

\begin{addmargin}[0.87cm]{0cm}

 \noindent \underline{Zero forcing color change rule:} If $u$ is blue and $v$ is the only white neighbor of $u$, then $u$ can force $v$ to be colored blue. If a vertex $u$ forces $v$, then we denote this by $u \rightarrow v$.

\end{addmargin}

\vspace{0.1in}

A zero forcing process may include multiple steps, called {\em time-steps}, and during each time-step multiple white vertices may become blue.  Once a vertex is blue, it will remain blue for the remainder of the zero forcing process.  During the process, the color change rule is applied repeatedly until either there are no white vertices remaining or no blue vertex has a unique white neighbor.   If $B$ is the set of vertices initially chosen to be blue and after a sufficient number of applications of the color change rule every vertex in $G$ is blue, then $B$ is a {\em zero forcing set} of $G$, otherwise $B$ is a {\em failed zero forcing set} of $G$.  The {\em zero forcing number} of $G$, denoted $\Z(G)$, is the minimum cardinality among the zero forcing sets of $G$.  If $B$ is a zero forcing set of $G$ and $\abs{B}=\Z(G)$, then we say that $B$ is a {\em minimum zero forcing set}.  During a zero forcing process, when there is more than one vertex which can force $v$, say $u_1$ and $u_2$, one must choose which vertex does the forcing.  Since multiple vertices may become blue during a single time-step it is sometimes convenient to say that a set of vertices $U$ forces another set of vertices $W$, denoted $U \rightarrow W$. Rigorously, this means that for all vertices $w \in W$ there is a vertex $u \in U$ such that $u \rightarrow w$.

\begin{figure}[h]
  \centering
  \includegraphics[width=.75\linewidth]{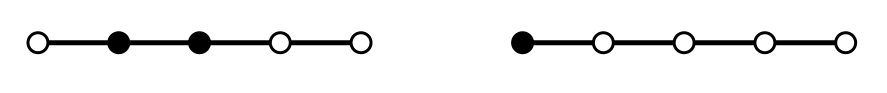}
  \caption{Examples of a zero forcing set (which is not minimum) and a minimum zero forcing set.}
  \label{Def2,2}
\end{figure}

In this paper we will be discussing multiple types of zero forcing processes.  The first and primary type of process is called a propagating family of forces.  A {\em propagating family of forces} $\mathcal F$ is an ordered collection of sets of forces $\{F^{(k)}\}_{k=1}^K$ such that at each time-step $k$ every white vertex that can become blue, according to the criteria outlined in the zero forcing color change rule, will become blue.    Let $B$ be a zero forcing set of a graph $G$ and $\mathcal F$ be a propagating family of forces of $B$ on $G$.  Define $B^{[0]}=B$, and for each integer $k \geq 1$, let $B^{(k)}$ denote the set of vertices which are forced during time-step $k$ and define $B^{[k]}=B^{(k)} \cup B^{[k-1]}$.  The {\em propagation time} of $B$ in $G$, denoted $\pt(G,B)$, is the least $k$ such that $B^{[k]}=V(G)$.  The {\em propagation time of $G$} is given by $\pt(G)=\min\{\pt(G,B):\abs{B}=\Z(G)\}$.  If $B$ is a minimum zero forcing set of $G$ with $\pt(G,B)=\pt(G)$, then $B$ is an {\em efficient zero forcing set}.

The second type of process is a relaxed chronology of forces.  To define this second type of process, introduced in \cite{PIPs}, first define $S(G,B)$ to be the collection of forces $u \rightarrow v$ such that $u \rightarrow v$ is a valid force according to the zero forcing color change rule when $B$ is the set of vertices which is currently blue.  A {\em relaxed chronology of forces} is an ordered collection of sets of forces, denoted $\mathcal F=\{F^{(k)}\}_{k=1}^K$, such that at each time-step $k$ if $E_{\mathcal F}^{[k-1]}$ denotes the set of vertices which are currently blue, then $F^{(k)} \subseteq S(G,E_{\mathcal F}^{[k-1]})$ such that for a given white vertex $v$ at most one force $u \rightarrow v \in F^{(k)}$, and $E_{\mathcal F}^{[K]}=V(G)$.  The {\em completion time} of a relaxed chronology of forces $\mathcal F=\{F^{(k)}\}_{k=1}^K$ is given by $\ct (\mathcal F)=K$.  For a zero forcing set $B$ and a relaxed chronology of forces $\mathcal F$, the sequence of sets $\{E_{\mathcal F}^{[k]}\}_{k=0}^K$ with $B=E_{\mathcal F}^{[0]} \subseteq E_{\mathcal F}^{[1]} \subseteq \dots \subseteq E_{\mathcal F}^{[K-1]} \subseteq E_{\mathcal F}^{[K]}=V(G)$, is called the {\em expansion sequence of $B$ induced by $\mathcal F$} and for each time-step $k$, $E_{\mathcal F}^{[k]}$ is called the $k$-th expansion of $B$ induced by $\mathcal F$. It is worth noting that if at each time-step $F^{(k)}$ is chosen to be maximal, then the relaxed chronology of forces $\mathcal F$ will be a propagating family of forces.  
Since the construction of a relaxed chronology allows for great flexibility in choosing the subset $F^{(k)} \subseteq S(G,E_{\mathcal F}^{[k-1]})$ at each time-step, it is often more convenient to utilize relaxed chronologies of forces during a proof.  However, since the aim of this paper will be to establish the propagation times of different graphs, it is important to note that for each minimum zero forcing set $B$ of $G$ and each relaxed chronology of forces $\mathcal F$ of $B$ on $G$, $\pt(G) \leq \ct(\mathcal F)$ and for at least one choice of $B$ and $\mathcal F$, $\pt(G)=\ct(\mathcal F)$.

For a given zero forcing set $B$ and a relaxed chronology of forces $\mathcal F=\{F^{(k)}\}_{k=1}^K$ of $B$ on $G$, a {\em forcing chain} induced by $\mathcal F$ is a sequence of vertices $(v_0,v_1,\dots, v_N)$ of $G$ such that $v_0 \in B$, $v_N$ does not force during $\mathcal F$, and for each $i$ with $0 \leq i \leq N-1$, $v_i \rightarrow v_{i+1} \in F^{(k)}$ for some $k$ with $0 \leq k \leq K$. The collection of forcing chains induced by $\mathcal F$ is a {\em chain set}.  Since $B$ is a zero forcing set of $G$, each vertex of $G$ is contained in some forcing chain.  Furthermore, since each vertex is forced by at most one other vertex and in turn performs at most one force during a relaxed chronology of forces, it follows that a chain set forms a path cover, providing the following result.

\begin{theorem}{\cite{param}}\label{Lemon}
	Let $G$ be a graph. Then $\p(G) \leq \Z(G)$. 
\end{theorem}

Vertices which are members of $B$ are said to be {\em initial} and vertices which do not perform a force during $\mathcal F$ are said to be {\em terminal}.  If a terminal vertex $v$ is initial, then $v$ is said to be {\em passively terminal}.  The set of vertices which are terminal are called the terminus of $\mathcal F$, denoted $\Term(\mathcal F)$.  The next lemma concerning the terminus appeared in \cite{PIPs} as a restatement of \cite[Theorem 2.6]{param} and a generalization of \cite[Observation 2.4]{proptime}, and will be of use later.

\begin{theorem}\label{term}\cite{PIPs, param, proptime}
Let $G$ be a graph, $B$ be a zero forcing set of $G$, and $\mathcal F$ be a relaxed chronology of forces of $B$ on $G$.  Then $\Term(\mathcal F)$ is a zero forcing set of $G$ and $\pt(G,\Term(\mathcal F)) \leq \pt(G, B)$.     
\end{theorem}

One of the most natural examples of an infinite class of finite graphs is the wheel graph.  The wheel graph on $n+1$ vertices, 
denoted $W_{n+1}$, is the graph with vertex set $V(W_{n+1})=\{v_i\}_{i=0}^n$ and edge set such that $v_0v_i \in E(W_{n+1})$ for each $i \in \{1,2,\dots,n\}$, $v_iv_{i+1} \in E(W_{n_+1})$ for each $i \in \{1,2,\dots,n-1\}$, and $v_1v_n \in E(W_{n+1})$.  Other classes of graphs may be built upon the structure of wheel graphs, such as helm graphs and gear graphs, and thus considered modifications of the wheel graph.  In the next two sections we will consider the zero forcing numbers and propagation times of gear graphs and helm graphs.  For the purposes of this paper, the definitions we use for helm graphs and gear graphs are more general than the traditional definitions.  However, in our calculations the zero forcing numbers and propagation times of the more traditional versions of these graph classes will also be addressed.

\section{Gear graphs} 

\noindent We first address gear graphs, and have included their definition below:

\begin{definition}
The {\em generalized gear graph}, denoted $Gr(m,r)$, is the graph with vertex set $V\big{(}Gr(m,r)\big{)}=\{v_i\}_{i=0}^{m(r+1)}$  such that for $v_i,v_j \in V\big{(}Gr(m,r)\big{)}$ distinct, $v_iv_j \in E\big{(}Gr(m,r)\big{)}$ if and only if one of the following is true:
    \begin{itemize}
 	\item $i=0$ and $(r+1)\big{|}j$.
 	\item $i,j \not =0$ and $\abs{i-j}=1$
 	\item $\{i,j\}=\big{\{}1,m(r+1)\big{\}}$.
    \end{itemize} 
\end{definition}

Note that in the special case when $r=1$, the above definition for the gear graph becomes the traditional gear graph, denoted $Gr_n$.  We would now like to establish the zero forcing number of our more generalized gear graph, denoted $\Z\big{(}Gr(m,r)\big{)}$.  Before preceding it is helpful to note that it was shown in \cite{paramlong} that

\begin{theorem}\label{mindeg}\cite{paramlong}
    Let $G$ be a graph, $B$ be a zero forcing set of $G$, $\mathcal F$ be a relaxed chronology of forces of $B$ on $G$.  If $v \in V(G)$ such that $v\rightarrow u \in F^{(1)}$ for some vertex $u \in V(G)$, then at least $\deg(v)-1$ neighbors of $v$ are in $B$.  Furthermore, $\Z(G) \geq \delta(G)$.  
\end{theorem}

\begin{figure}[h]
  \begin{center}
  \includegraphics[width=0.4\linewidth]{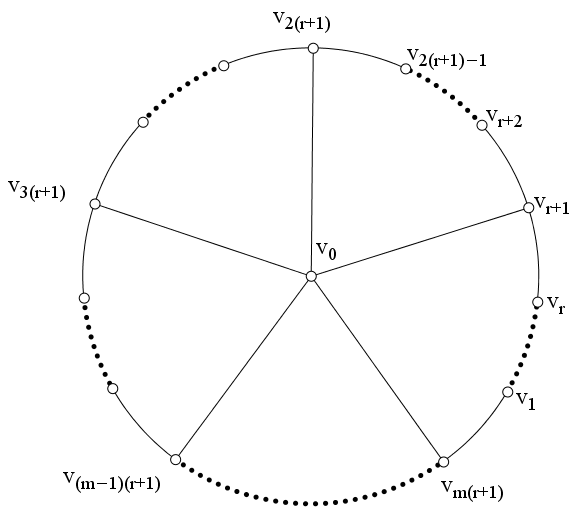}
  \caption{Generalized gear graph $G(m,r)$}
  \label{GenGear}
  \end{center}
\end{figure}

For the purpose of the following proofs, when discussing these graphs, we will sometimes refer to $v_0$ as the center vertex, any $v_i$ such that $(r+1)\big{|}i$ as a spoke vertex, and any vertex which is neither a center vertex nor a spoke vertex as an intermediate vertex. Even though the intermediate vertices and spoke vertices are only indexed up to $m(r+1)$, for certain arguments it will allow the discussion to remain more direct and streamlined if it is understood that for a vertex $v_M$, with $M > m(r+1)$, the writers are referencing a vertex $v_N$, where $N = M \bmod (mr+m)$.

\begin{theorem}
	$\Z\big{(}Gr(m,r)\big{)} = 3$.
\end{theorem}

\begin{proof}
	Since $\delta(Gr(m,r)) =2$, $\Z\big{(}Gr(m,r)\big{)} \geq 2$ by Theorem \ref{mindeg}.  Next, we would like to show $\Z \big( Gr(m,r) \big) \neq 2$. To this end, suppose by way of contradiction that there exists a zero forcing set of $Gr(m,r)$, say $B$, such that $\abs{B} =2$. When finding possible zero forcing sets, it is important to consider the degree of each vertex. In this case, $\deg(v_0) = m$, $\deg(v_{(r+1)k}) = 3$ where $k\in \{1,2,...,m\}$, and $\deg(v_{i}) = 2$ where $(r+1)\not | i$. It is clear that each vertex has degree greater than or equal to $2$. We know then, that in order for any forcing to occur with the vertices in $B$, the two vertices must be adjacent. There are only three possible configurations which satisfy this requirement. 
	
	Configuration 1: We consider $B= \{v_0,v_{(r+1)k}\}$. Note, $\deg(v_0) = m$ and $\deg(v_{(r+1)k}) = 3$. Therefore, since $v_0$ and $v_{(r+1)k}$ each have degree of at least three and thus two white neighbors, no forcing can occur and $B$ is a failed zero forcing set. 
	
	Configuration 2: Up to isomorphism, we consider $B = \{v_{(r+1)k}, v_{(r+1)k-1}\}$ for some $k$. Note that $\deg(v_{(r+1)k-1}) = 2$ and one of its adjacent neighbors, namely $v_{(r+1)k}$, is in $B$. Thus on time-step one, $v_{(r+1)k-1} \rightarrow v_{(r+1)k-2}$. After time-step $r$, however, no forcing can occur as $\deg(v_{(r+1)k}) = \deg(v_{(r+1)(k-1)}) = 3$ and each have two white neighbors. In particular, $v_{(r+1)k}$ is neighbored by $v_{(r+1)k+1}$ and $v_{0}$ which are both white, and $v_{(r+1)(k-1)}$ is neighbored by $v_{(r+1)(k-1)-1}$ and $v_0$ which are both white.  Thus $B$ is a failed zero forcing set.
	
Configuration 3: We consider $B= \{v_i,v_{i+1} \}$, where $r+1$ divides neither $i$ nor $i+1$. Note that since $\deg(v_i)=\deg(v_{i+1})=2$, on time-step one, $v_i\rightarrow v_{i-1}$ and $v_{i+1} \rightarrow v_{i+2}$.  This process can continue until the spoke vertices $v_{(r+1)k}$ and $v_{(r+1)(k+1)}$ are blue. However, since $\deg(v_{(r+1)(k+1)})=\deg(v_{(r+1)k})=3$, and they each only have one blue neighbor no more forcing can occur.  Thus $B$ is a failed zero forcing set.
	
	Since $Gr(m,r)$ has no possible zero forcing sets of size 2, we have established that $\Z\big{(}Gr(m,r)\big{)} \geq3$. To show that $\Z\big{(}Gr(m,r)\big{)} = 3$ we must find a zero forcing set of $Gr(m,r)$ of size 3. We will now construct such a zero forcing set. 
 Let $B = \{v_0, v_{m(r+1)}, v_{1}\}$. We define a relaxed chronology of forces $\mathcal F$ by requiring $v_0$ to remain passively terminal and allowing $v_{m(r+1)}$ and $v_{1}$ to initiate a forcing process along the cycle of spoke and intermediate vertices; in particular, on time-step $k$, $v_k \rightarrow v_{k+1}$ and $v_{m(r+1)-k+1} \rightarrow v_{m(r+1)-k}$. This process continues until every vertex of $Gr(m,r)$ is blue, and thus we have that  $\Z\big{(}Gr(m,r)\big{)} = 3$.
\end{proof}

\setcounter{case}{0}

\begin{figure}[h]
  \begin{center}
  \includegraphics[width=0.4\linewidth]{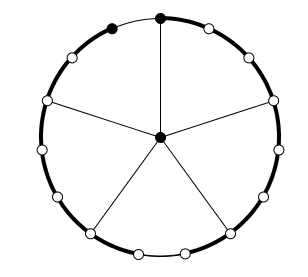}
  \caption{A zero forcing set and a chain set of $Gr(5,2)$ with chain edges in bold}
  \label{Gearchain}
  \end{center}
\end{figure}

We now look to the propagation time of gear graphs.

\begin{theorem}
$\text{~}$
\begin{itemize}
	\item In the case that $m$ is odd and $r=1$, $\pt\big{(}Gr(m,r)\big{)}=m-1$. \item Otherwise, $\pt\big{(}Gr(m,r)\big{)}=\big{\lceil} \frac{ m (r+1)}{2} \big{\rceil} -2.$
\end{itemize}
\end{theorem}

\begin{figure}
    \centering
    \begin{subfigure}[b]{0.24\textwidth}
    \centering
    \includegraphics[width=\textwidth]{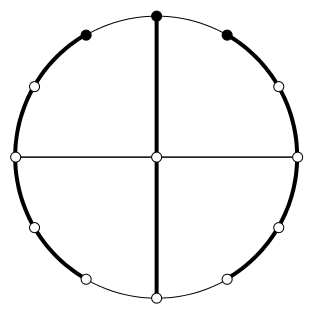}
    \caption{Case 1}
    \end{subfigure}
    \begin{subfigure}[b]{0.24\textwidth}
    \centering
    \includegraphics[width=\textwidth]{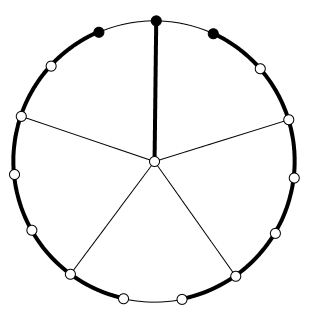}
    \caption{Case 2}
    \end{subfigure}
    \begin{subfigure}[b]{0.24\textwidth}
    \centering
    \includegraphics[width=\textwidth]{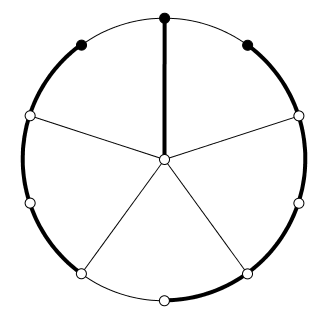}
    \caption{Case 3 [r=1]}
    \end{subfigure}
    \begin{subfigure}[b]{0.24\textwidth}
    \centering
\includegraphics[width=\textwidth]{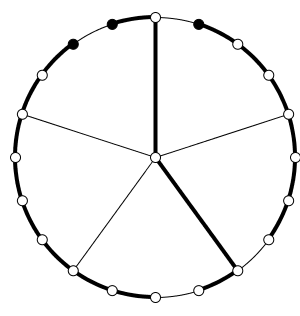}
    \caption{Case 3 [r$>$1]}
    \end{subfigure}
\end{figure}

\begin{proof}
	To determine the propagation time of the generalized gear graph, we will consider all possible zero forcing sets of size 3 and show that one of them produces the most efficient propagation time. Due to the degree of all the vertices being at least two, we need at least two vertices in our zero forcing set to be adjacent. There are only three possible ways to have two adjacent vertices on the gear graph: a spoke vertex and an intermediate vertex, two intermediate vertices, and a spoke vertex and a  center vertex. This leaves one vertex left to be chosen for the zero forcing set. Depending on the chosen vertex, either forcing will fail after at most $r$ time-steps after the second spoke vertex has been reached, or up to isomorphism, we will get two families of zero forcing sets: $B_1=\{v_0,v_i,v_{i+1}\}$ and $B_2=\{v_1,v_i,v_{i+1}\}$ where  
	\[i\in S= \{(m-1)(r+1),(m-1)(r+1)+1,...,m(r+1)-1\}.\]
	
    Since we only have three elements in our zero forcing set, on any given time-step we can force at most three vertices. Since two of our forcing chains will be passing along the cycle of spoke and intermediate vertices, we can only force three vertices when the center vertex is being forced or when the center vertex is forcing. Thus, this can happen at most twice. Thus, if we can find a relaxed chronology of forces where three vertices are forced twice, and there are no time-steps where only a single vertex is forced, then we will have found our propagation time. When $m$ is odd and $r$ is even we have an even number of vertices in our graph, so due to parity we will only be able to force three vertices once.  Since $\big{|}V\big{(}Gr(m,r)\big{)}\big{|}=m(r+1)+1$, in either case we have that 
    
    \[\pt\big{(}Gr(m,r)\big{)} \geq \left\lceil\frac{m(r+1)+1-3-6}{2}\right\rceil+2=\left\lceil\frac{m(r+1)}{2}\right\rceil-2.\]

    We will show that this bound is attained by a zero forcing set anytime that either $r>1$ or $m$ is even.  However, in the special case where $r=1$ and $m$ is odd, this will require proof by exhaustion.
    
    \begin{case}
    $m$ is even.
    \end{case}
    
    Let our zero forcing set be $B=\{v_1,v_{m(r+1)},v_{m(r+1)-1}\}$. On time-step 1, $v_{m(r+1)}\rightarrow v_0$, $v_1 \rightarrow v_2$, and $v_{m(r+1)-1} \rightarrow v_{m(r+1)-2}$. On time-step $k$, $v_k \rightarrow v_{k+1}$ and $v_{m(r+1)-k} \rightarrow v_{m(r+1)-k-1}$. Since we have taken $v_{m(r+1)}$ in our zero forcing set, and we force an even number of spoke vertices on each time-step, on the final time-step we are left with a single unforced spoke vertex. Because of this, on the final time-step, three vertices are able to force: $v_0 \rightarrow v_{\frac{m}{2}(r+1)}$, $v_{\frac{m}{2}(r+1)-2} \rightarrow v_{\frac{m}{2}(r+1)-1}$, and $v_{\frac{m}{2}(r+1)+2} \rightarrow v_{\frac{m}{2}(r+1)+1}$. Since we have taken three vertices in our zero forcing set, we were able to force three vertices on each of two time-steps, and two vertices on the remaining time-steps, we have that \[\pt\big{(}Gr(m,r)\big{)}=\pt\big{(}Gr(m,r),B\big{)}=\frac{m(r+1)+1-3-6}{2}+2=\frac{m(r+1)}{2}-2=\left\lceil\frac{m(r+1)}{2}\right\rceil-2.\]

    \begin{case}
    $m$ is odd, $r$ is even
    \end{case}
    \indent Let our zero forcing set be $B=\{v_1,v_{m(r+1)},v_{m(r+1)-1}\}$. On time-step 1, $v_{m(r+1)}\rightarrow v_0$, $v_1 \rightarrow v_2$, and $v_{m(r+1)-1} \rightarrow v_{m(r+1)-2}$. After the first time-step, there will be an even number of vertices left, so it will not be possible to have another time-step where three vertices are forced without having a time-step where only a single vertex is forced. At this point, all forcing will be done along the cycle of spoke and intermediate vertices with two vertices being forced at each time-step. Since we have taken three vertices in our zero forcing set, we were able to force three vertices on one time-step, and two on each remaining time-step, we have that 
    \[\pt\big{(}Gr(m,r)\big{)}=\pt\big{(}Gr(m,r),B\big{)}=\frac{m(r+1)+1-3-3}{2}+1=\frac{m(r+1)+1}{2}-2=\left\lceil\frac{m(r+1)}{2}\right\rceil-2.\]
    
    \begin{case}
    $m$ is odd, $r$ is odd
\end{case}

\vspace{-0.175in}

\begin{addmargin}[0.2in]{0in}
\begin{subcase} 
$r>1$
\end{subcase}

    \vspace{0.1in}
    
    Let our zero forcing set be $B=\{v_1,v_2,v_{m(r+1)-1}\}$. On time-step one, $v_1\rightarrow v_{m(r+1)}$ and $v_2\rightarrow v_3$. On time-step two, $v_3\rightarrow v_4$, $v_{m(r+1)}\rightarrow v_0$, and $v_{m(r+1)-1}\rightarrow v_{m(r+1)-2}$. On time-step $k$, $v_{k+1}\rightarrow v_{k+2}$ and $v_{m(r+1)-k+1}\rightarrow v_{m(r+1)-k}$. On time-step $\frac{m-1}{2}(r+1)-2$, the second to last spoke vertex, $v_{\frac{m-1}{2}(r+1)}$, is forced. Thus, on time-step $\frac{m-1}{2}(r+1)-1$, $v_0\rightarrow v_{\frac{m+1}{2}(r+1)}$, $v_{\frac{m-1}{2}(r+1)}\rightarrow v_{\frac{m-1}{2}(r+1)+1}$, and $v_{\frac{m+1}{2}(r+1)+2}\rightarrow v_{\frac{m+1}{2}(r+1)+1}$. On each remaining time-step, since $r$ is odd, two vertices are forced. Since we have taken three vertices in our zero forcing set, there are two time-steps on which we force three vertices, and two vertices are forced on every remaining time-step, we have that 
\end{addmargin}    

\[\pt\big{(}Gr(m,r)\big{)}=\pt\big{(}Gr(m,r),B\big{)}=\frac{m(r+1)+1-3-6}{2}+2=\frac{m(r+1)}{2}-2=\left\lceil\frac{m(r+1)}{2}\right\rceil-2.\]    

\vspace{-0.1in}

\begin{addmargin}[0.2in]{0in}
 \begin{subcase}  $r=1$
 \end{subcase}

\vspace{0.1in}
 
    First note that since we begin the process with an even number of white vertices, this is enough to guarantee that on average at most two vertices are forced during each time-step, and since $\big{|}V\big{(}G(m,r)\big{)}\big{|}=2m+1$ it follows that
    \[\pt\big{(}Gr(m,1)\big{)} \geq \frac{2m+1-3}{2}=m-1.\]
    Up to isomorphism, there are only three zero forcing sets of $Gr(m,1)$: $B_1=\{v_0, v_1,v_{2m}\}$, $B_2=\{v_1,v_{2m},v_{2m-1}\}$, and $B_3=\{v_1,v_3,v_{2m}\}$. Note that if $B_3$ is chosen as the zero forcing set, then on time-step one only a single vertex is forced, namely $v_2$. Next, note that if $B_1$ is chosen as the zero forcing set, then $v_0$ is already in our zero forcing set and thus there cannot be two time-steps where three vertices are forced.  We will now show that $B_2$ can force three vertices on the first time-step, one vertex on the final time-step, and two on every remaining time-step, and thus is an efficient zero forcing set. 
    
    When $B_2$ is chosen as the zero forcing set, on time-step 1, $v_1\rightarrow v_2$, $v_{2m}\rightarrow v_0$, and $v_{2m-1}\rightarrow v_{2m-2}$. At this point, all forcing will be done along the cycle of spoke and intermediate vertices with two vertices being forced at each time-step, except the final time-step when there is only a single vertex remaining. Since we have taken three vertices in our zero forcing set, we were able to force three vertices on time-step one, one vertex on the final time-step, and two vertices on each remaining time-step we have that 
\end{addmargin}

    \[\pt\big{(}Gr(m,r)\big{)}=\pt\big{(}Gr(m,1),B_2\big{)}=\frac{2m+1-3-3-1}{2}+2=m-1.\]

\end{proof} 

\section{Helm graphs}

We now consider helm graphs, for which a definition is provided below.

\begin{definition}
	 The helm graph, denoted $H(m,s)$, is the graph with vertex set $V\big{(}H(m,s)\big{)}=\{v_i\}_{i=0}^m \cup \{p_{i,j}\}_{i=1,}^m{}_{j=1}^s$ such that for $u_1,u_2 \in V\big{(}H(m,s)\big{)}$ distinct, $u_1u_2 \in E\big{(}H(m,s)\big{)}$ if and only if one of the following is true:
	\begin{itemize}
		\item $u_1 = v_0$ and $u_2 = v_{i}$ 
		\item $u_1=v_{i_1}$ and $u_2=v_{i_2}$ and in addition one of the following is also true:
		\begin{itemize}
			\item $\abs{i_1-i_2}=1$
			\item $\{i_1,i_2\}=\{1,m\}$
		\end{itemize}
		\item $u_1=v_i$ and $u_2=p_{i,j}$ for some $i \in \{1,2,...m\}$ and some $j \in \{1,...,s\}$.
	\end{itemize}
\end{definition}

\begin{figure}[h]
    \centering
    \includegraphics[width=.35\linewidth]{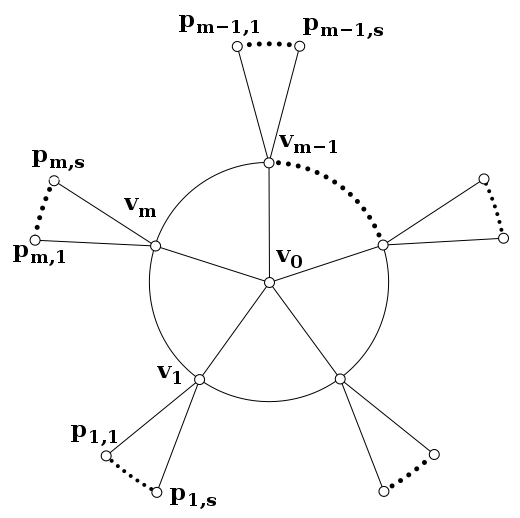}
    \caption{Generalized helm graph $H(m,s)$}
    \label{GenHelm}
\end{figure} 

For the duration of this paper we will take the convention that sets of vertices of the form $\{v_{f(j)}\}_{j=n}^m$, where $m<n$ and $f$ is some function on the natural numbers, will be considered empty.  As with gear graphs, we will sometimes refer to $v_0$ as the center vertex, and the remaining vertices which are neither pendant vertices nor the center vertex will be referred to as spoke vertices.  Again, even though spoke vertices are only indexed up to $m$, for certain arguments it will allow the discussion to remain more direct and streamlined if it is understood that for a vertex $v_M$, with $M > m$, the writers are referencing a vertex $v_N$, where $N = M \bmod m$.

Due to the nature of helm graphs, pendant vertices play a key role in zero forcing, so it will be beneficial to have extra terminology concerning them.

\begin{definition}
Let $m \geq 3$ be an integer and $B$ be the set of vertices of $H_m$ initially chosen to be blue. Two pendant vertices $p_i$ and $p_j$, are {\em consecutive} if their indices, $i$ and $j$, are consecutive natural numbers (modulo $m$). Pendant vertices that are members of $B$ are {\em forcing pendants}, and those which are not are {\em terminal pendants}.  Furthermore, if $p_j$, $p_{j+1}$, and $p_{j+2}$ are consecutive pendant vertices, but only $p_{j+1}$ is a forcing pendant, then $p_{j+1}$ is an {\em isolated forcing pendant}.  Maximal collections of consecutive forcing pendants are {\em groups of forcing pendants}, that is if $\{p_i\}_{i=j}^k$ is a set of consecutive forcing pendants, but $p_{j-1}$ and $p_{k+1}$ are terminal pendants, then $\{p_i\}_{i=j}^k$ is a group of forcing pendants.  Similarly, maximal collections of consecutive terminal pendants are {\em groups of terminal pendants}.   
\end{definition} 

Note that in the special case where $s=1$, the definition we provide for the helm graph becomes the traditional helm graph, denoted $H_n$. As we will see later, the case where $s>1$ is simpler to address for both $\Z(H(m,s))$ and $\pt(H(m,s))$.  For now we establish the results for the case when $s=1$.

\begin{remark}
It can be confirmed via exhaustion that $\Z(H_3)=3, \Z(H_4)=3, \pt(H_3)=2$, and $\pt(H_4)=3$.
\end{remark}

\begin{lemma}\label{helmlem}
For $m \geq 4$, $\Z(H_m) \geq \left\lceil \frac{m}{2} \right\rceil$.
\end{lemma}

\begin{proof}
Let $\mathcal Q$ be a path cover of $H_m$.  Note that $m$ is the number of pendant vertices of $H_m$. Since each path $Q \in \mathcal Q$ may contain at most two pendant vertices, we have $\lceil \frac{m}{2} \rceil \leq \abs{\mathcal Q}$.  Since $\mathcal Q$ was an arbitrary path cover, it follows that $\lceil \frac{m}{2} \rceil \leq \p(H_m)$.  Finally, by Theorem \ref{Lemon}, $\Z(H_m) \geq \p(H_m) \geq \left\lceil \frac{m}{2} \right\rceil$.
\end{proof}

Since $ \lceil \frac{m}{2} \rceil \leq \Z(H_m)$, to establish the value of $\Z(H_m)$ it remains to show that for $m>4$, $\Z(H_m) \leq \lceil \frac{m}{2} \rceil$. To prove that $ \Z(H_m) \leq \lceil \frac{m}{2} \rceil$, we need only to find one set of vertices $B$ such that $\abs{B}= \lceil \frac{m}{2} \rceil$, and demonstrate a relaxed chronology of forces $\mathcal F=\{F^{(k)}\}_{k=1}^K$ of $B$ on $H_m$. Furthermore, this will also give us an upper bound for the propagation time of $H_m$, namely $\pt(H_m) \leq K$.

\begin{theorem}\label{upper_helm} Let $m \geq 5$, then $\Z(H_m)=\lceil \frac{m}{2} \rceil$ and
		$$ \pt(H_m) \leq \begin{cases} 
	6 & \text{if }m \bmod 4 \equiv 0\\
	4 &  \text{if }m \bmod 4 \equiv 1\\
	5 &  \text{if }m \bmod 4 \equiv 2 \text{ or } m \bmod 4 \equiv 3.
	\end{cases} $$ 
\end{theorem} 

\begin{figure}[h]
\begin{center}
\includegraphics[width=1.005\linewidth]{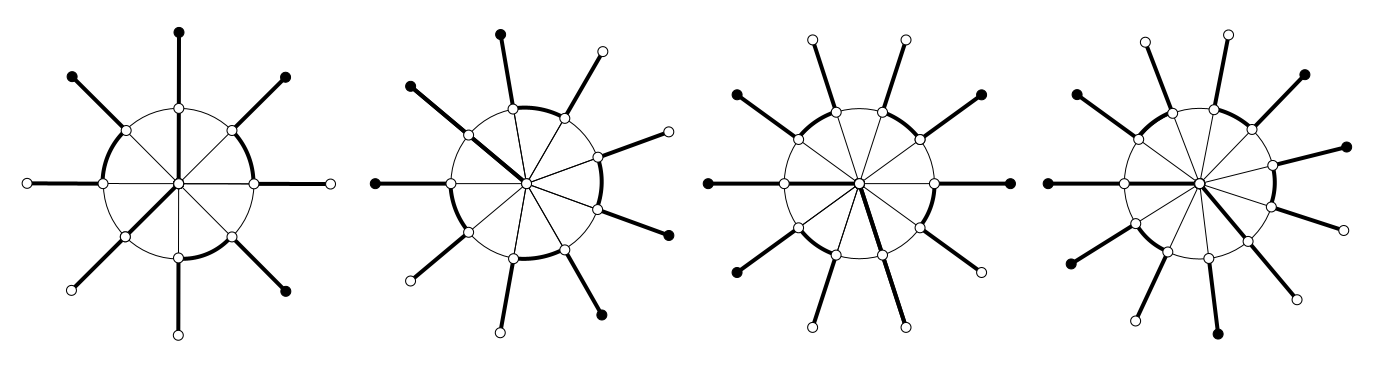}
  \caption{Zero forcing sets and chain sets of helm graphs.}
  \label{Helmchain}
  \end{center}
\end{figure}

\setcounter{case}{0}

\begin{proof}
	Let $m \geq 5$. Due to the nature of the helm graph, the propagation times are dependent on the congruence class $ m \bmod 4$, so we consider the following cases:
	\begin{case} $m \bmod 4 \equiv 0$
	\end{case}
		Let $B = \{ p_i \}_{i=1}^3 \cup \{p_{6+4j}\}_{j=0}^{\frac{m-12}{4}} \cup \{p_{7+4j}\}_{j=0}^{\frac{m-12}{4}} \cup \{p_{m-3}\} $. Note since $m$ is even, $\left\lceil \frac{m}{2} \right\rceil = \frac{m}{2}$ and 
  
  \[\abs{B}= 3 + 2\left( \frac{m-12}{4}+1\right ) +1 = \frac{m-12}{2} + 6 = \frac{m-12}{2} + \frac{12}{2} = \frac{m}{2}.\]
  
  Therefore, we know that if $B$ is a zero forcing set, then it is of minimum size. 
 We will now construct a relaxed chronology of forces $\mathcal F=\{F^{(k)}\}_{k=1}^6$. On time-step 1, the pendant vertices in $B$ perform forces as follows: $\{p_i\}_{i=1}^3 \rightarrow \{v_i\}_{i=1}^3$, $\{p_{6+4j}\}_{j=0}^{\frac{m-12}{4}} \rightarrow \{v_{6+4j}\}_{j=0}^{\frac{m-12}{4}}$, $\{p_{7+4j}\}_{j=0}^{\frac{m-12}{4}} \rightarrow \{v_{7+4j}\}_{j=0}^{\frac{m-12}{4}}$, and $p_{m-3} \rightarrow v_{m-3}$. On time-step 2, $v_2 \rightarrow v_0$ since $v_0$ was its only white neighbor ($p_2 \in B$ and $v_1,v_3$ were forced on time-step 1). On time-step 3, $v_1 \rightarrow v_m$ since $v_2, p_1,$ and  $v_0$ have been forced and are the only neighbors of $v_1$ besides $v_m$. Similarly, on time-step 3, $v_3 \rightarrow v_4$, and continuing in this manner, $\{v_{6+4j}\}_{j=0}^{\frac{m-12}{4}}\rightarrow \{v_{5+4j}\}_{j=0}^{\frac{m-12}{4}}$ since the only neighbors of $\{v_{6+4j}\}_{j=0}^{\frac{m-12}{4}}$ are $\{p_{6+4j}\}_{j=0}^{\frac{m-12}{4}}$, $\{v_{7+4j}\}_{j=0}^{\frac{m-12}{4}}$, and $v_0$ all of which are in $B$ or have been forced already. In addition $\{v_{7+4j}\}_{j=0}^{\frac{m-12}{4}}\rightarrow \{v_{8+4j}\}_{j=0}^{\frac{m-12}{4}}$ since the only neighbors of $\{v_{7+4j}\}_{j=0}^{\frac{m-12}{4}}$ are $\{p_{7+4j}\}_{j=0}^{\frac{m-12}{4}}$, $\{v_{6+4j}\}_{j=0}^{\frac{m-12}{4}}$, and $v_0$ all of which are in $B$ or have been forced already. At this point, all vertices along the cycle of spoke vertices have been forced except for $v_{m-2}$ and $v_{m-1}$. On time-step 4, $v_{m-3} \rightarrow v_{m-2}$ since its only other neighbors are $p_{m-3}$ (which was in $B$), $v_{m-4}$ (forced on time-step 3 by $v_{m-5}$), and $v_0$ (forced on time-step 2). Now all spoke vertices have been forced except for $v_{m-1}$ which can be forced on time-step 5 by $v_0$ (since all but one of $v_0$'s neighbors are blue). Finally, the only white vertices remaining are pendant vertices. Therefore, on time-step 6, their corresponding spoke vertices may force them blue, thus completing the forcing.

	\begin{case} $m \bmod 4 \equiv 1$
	\end{case}
		 Let $B = \{ p_i \}_{i=1}^3 \cup \{p_{6+4j}\}_{j=0}^{\frac{m-9}{4}} \cup \{p_{7+4j}\}_{j=0}^{\frac{m-9}{4}} $. Note since $m$ is odd, $\left\lceil \frac{m}{2} \right\rceil = \frac{m+1}{2}$ and
   
		 \[\abs{B}= 3 + 2\left( \frac{m-9}{4}+1\right) = \frac{m-9}{2} + 5 = \frac{m-9}{2} + \frac{10}{2} = \frac{m+1}{2}.\]
   
		We force in a manner similar to Case 1. On the first time-step $\{ p_i\}_{i=1}^3 \rightarrow \{v_i\}_{i=1}^3$, $\{p_{6+4j}\}_{j=0}^{\frac{m-9}{4}} \rightarrow \{v_{6+4j}\}_{j=0}^{\frac{m-9}{4}}$, $\{p_{7+4j}\}_{j=0}^{\frac{m-9}{4}} \rightarrow \{v_{7+4j}\}_{j=0}^{\frac{m-9}{4}}$. On time-step 2, $v_2 \rightarrow v_0$. On time-step 3, $v_1 \rightarrow v_m$, $v_3 \rightarrow v_4$, $\{v_{6+4j}\}_{j=0}^{\frac{m-9}{4}} \rightarrow \{v_{5+4j}\}_{j=0}^{\frac{m-9}{4}}$, and $\{v_{7+4j}\}_{j=0}^{\frac{m-9}{4}} \rightarrow \{v_{8+4j}\}_{j=0}^{\frac{m-9}{4}}$. At this point, all vertices in the graph have been forced except $\{p_{5+4j}\}_{j=0}^{\frac{m-9}{4}}$, $\{p_{8+4j}\}_{j=0}^{\frac{m-9}{4}}$, $p_4$, and $p_m$. These can all be forced on time-step 4 since they are the only white neighbors of their respective spoke vertices, $\{v_{5+4j}\}_{j=0}^{\frac{m-9}{4}}$, $\{v_{8+4j}\}_{j=0}^{\frac{m-9}{4}}, v_4$ and $v_m$.

	\begin{case} $m \bmod 4 \equiv 2$ 
	\end{case}
			Let $B = \{ p_i \}_{i=1}^3 \cup \{p_{6+4j}\}_{j=0}^{\frac{m-10}{4}} \cup \{p_{7+4j}\}_{j=0}^{\frac{m-10}{4}} $. Note since $m$ is even, $\left\lceil \frac{m}{2} \right\rceil = \frac{m}{2}$ and
   
   \[\abs{B}= 3 + 2\left( \frac{m-10}{4}+1\right) = \frac{m-10}{2} + 5 = \frac{m-10}{2} + \frac{10}{2} = \frac{m}{2}.\]
   
   The first three time-steps are again similar to Case 1. After time-step 3, the only white spoke vertex is $v_{m-1}$, which can be forced on time-step 4 by $v_0$. On time-step 5 the only remaining white vertices are pendant vertices which can be forced by their respective spoke vertices.
		
	\begin{case} $m \bmod 4 \equiv 3$
	\end{case}
	    Let $B= \{ p_i \}_{i=1}^3 \cup \{p_{6+4j}\}_{j=0}^{\frac{m-11}{4}} \cup \{p_{7+4j}\}_{j=0}^{\frac{m-11}{4}} \cup \{p_{m-1}\} $. Note since $m$ is odd, $\left\lceil \frac{m}{2} \right\rceil = \frac{m+1}{2}$ and 
     
     \[\abs{B}= 3 + 2\left( \frac{m-11}{4}+1\right ) +1 = \frac{m-11}{2} + 6 = \frac{m-11}{2} + \frac{12}{2} = \frac{m+1}{2}.\]
     
     Again, the first three time-steps are similar. At this point, all vertices along the cycle of spoke vertices have been forced except for $v_{m-2}$, which can be forced on time-step 4 by $v_0$. Finally, the only white vertices remaining are pendant vertices. Therefore, on time-step 5, their corresponding spoke vertices may force them blue, thus completing the forcing. 
	    
	    In each case, we have constructed a zero forcing set $B$ and a corresponding relaxed chronology of forces $\mathcal F$, thus showing
     \begin{equation*} \Z(H_m) \leq \Big{\lceil} \frac{m}{2} \Big{\rceil} \text{~and~} \pt(H_m) \leq \ct(\mathcal F) \leq \begin{cases} 
	6 & \text{if }m \bmod 4 \equiv 0\\
	4 &  \text{if }m \bmod 4 \equiv 1\\
	5 &  \text{if }m \bmod 4 \equiv 2 \text{ or } m \bmod 4 \equiv 3.
	\end{cases}
 \end{equation*}
	
	\noindent Furthermore, combining this with Lemma \ref{helmlem}, we have shown $\Z(H_m) = \left\lceil \frac{m}{2} \right\rceil$.
\end{proof}

Our next goal will be to determine lower bounds for $\pt(H_m)$, thus establishing the propagation time.  However, we must first provide some preliminary results to be used in said theorem.

\begin{observation}\label{Pendant_Vertex}
Let $B$ be a minimum zero forcing set of $H_m$, $\mathcal F$ be a relaxed chronology of forces of $B$ on $H_m$, and $\mathcal C$ be the chain set induced by $\mathcal F$.  If $m$ is even, then each forcing chain in $\mathcal C$ will contain two pendant vertices. If $m$ is odd, then all but one forcing chain in $\mathcal C$ will contain two pendant vertices with the remaining forcing chain containing one pendant vertex.
\end{observation}

We know this observation to be true since a forcing chain, being a vertex induced path, can have at most two pendant vertices of $H_m$. If $m$ is even, there will be $\frac{m}{2}$ chains. Since there are $m$ pendant vertices, each forcing chain must contain two pendant vertices. If $m$ is odd, there will be $\frac{m+1}{2}$ chains. Since there are $m$ pendant vertices, all but one chain must contain two pendant vertices, the remaining chain will have exactly one pendant vertex.

\begin{lemma}\label{Three_Consecutive}
If $m$ is even, then there can not be two distinct sets of three consecutive pendant vertices in a minimum zero forcing set of $H_m$. However, any minimum zero forcing set must contain a group of $3$ forcing pendants. 
\end{lemma}

\begin{figure}[h]
    \centering
    \includegraphics[width=.4\linewidth]{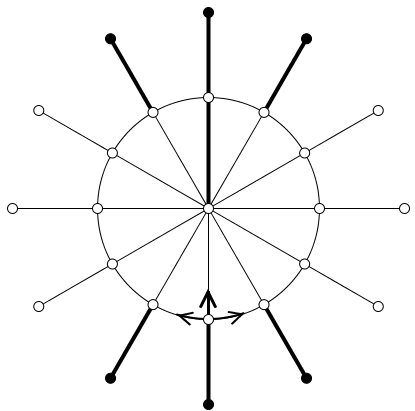}
    \caption{An example of $H_{12}$ with two sets of three consecutive forcing pendants.}
    \label{Lem3,7}
\end{figure}

\begin{proof}
Suppose by way of contradiction that $B$ is a minimum zero forcing set of $H_m$ that contains two sets of three consecutive pendant vertices, $\{p_{m(1,i)}\}_{i=1}^3$ and $\{p_{m(2,i)}\}_{i=1}^3$, with $m(j,i)+1=m(j,i+1)$. Every vertex in $B$ must force in order for its forcing chain to contain two pendant vertices. In particular, each such pendant vertex can only force its neighboring spoke vertex, so for each pair $(j,i)\in\{1,2\} \times\{1,2,3\}$, $p_{m(j,i)}\rightarrow v_{m(j,i)}$. Since $v_{m(1,2)}$ and $v_{m(2,2)}$ each only have one remaining white neighbor, that being the center vertex, and the center vertex can only be a member of one forcing chain, either the forcing chain beginning at $p_{m(1,2)}$ or the forcing chain beginning at $p_{m(2,2)}$ contain a single pendant vertex which contradicts Observation \ref{Pendant_Vertex}. 

Now suppose $B$ is a minimum zero forcing set of $H_m$ but $B$ does not contain three consecutive pendant vertices. After each pendant vertex in $B$ forces its neighboring spoke vertex, each such spoke vertex will have at least one white neighbor on the cycle of spoke vertices. Since each of the spoke vertices are adjacent to the center vertex and the center vertex has not been forced, these vertices have two white neighbors and thus no further forcing can occur.
\end{proof}

\begin{lemma}\label{Three_Consecutive_Odd}
If $m$ is odd, then there cannot be three sets of three consecutive pendant vertices in a minimum zero forcing set of $H_m$. In addition, if every vertex in a minimum zero forcing set $B$ of $H_m$ is a pendant vertex, then $B$ must contain three consecutive pendant vertices. 
\end{lemma}

\begin{figure}[h]
    \centering
    \includegraphics[width=.4\linewidth]{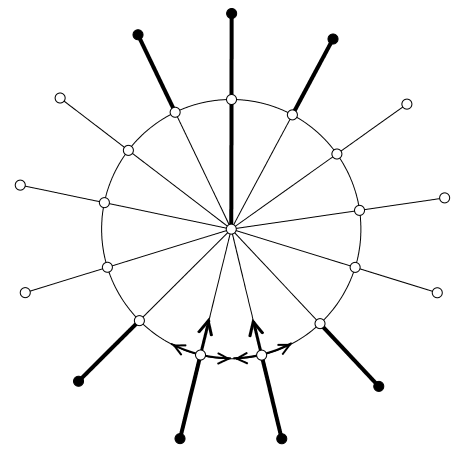}
    \caption{An example of $H_{13}$ with three sets of three consecutive forcing pendants.}
    \label{Lem3,8}
\end{figure}

\begin{proof}
Suppose by way of contradiction that we have a minimum zero forcing set $B$ that contains three sets of three consecutive pendant vertices, $\{p_{m(1,i)}\}_{i=1}^3$, $\{p_{m(2,i)}\}_{i=1}^3$, and $\{p_{m(3,i)}\}_{i=1}^3$, with $m(j,i)+1=m(j,i+1)$. Due to Observation \ref{Pendant_Vertex}, all but one vertex in $B$ has to force, otherwise it can not be the case that all but one of their forcing chains contain two pendant vertices. In particular, each of those pendant vertices can force at most their neighboring spoke vertex, so for each pair $(j,i)\in\{1,2,3\} \times\{1,2,3\}$ where forcing occurs, $p_{m(j,i)}\rightarrow v_{m(j,i)}$. Then at most one forcing chain beginning at a pendant vertex $p_{m(j,i)})$ will not contain the neighboring vertex $v_{m(j,i)}$. Due to this, for at least two of the $v_{m(j,2)}$, every path of white vertices beginning at a neighbor of $v_{m(j,2)}$ and terminating at a pendant vertex will pass through the center vertex. However, the center vertex can only be a member of one forcing chain, leaving a total of two forcing chains each of which only contain a single pendant vertex, thus contradicting Observation \ref{Pendant_Vertex}. 

Suppose now we have a minimum zero forcing set $B$ in which every vertex is a pendant vertex, but $B$ does not contain a set of three consecutive pendant vertices. As in Lemma \ref{Three_Consecutive}, the forcing process cannot go further than the spoke vertices neighboring the members of $B$, since each of the spoke vertices will have two white neighbors, specifically the center vertex and a spoke vertex.
\end{proof}

\setcounter{case}{0}

\begin{theorem}
For $m \geq 5$
\begin{equation*}
    \pt(H_m) = 
    \begin{cases}
    6 \indent & \normalfont{\text{if }}m \bmod 4 \equiv 0 \\
    4 & \normalfont{\text{if }}m \bmod 4 \equiv 1 \\
    5 & \normalfont{\text{if }}m \bmod 4 \equiv 2 \\ 
    5 & \normalfont{\text{if }}m \bmod 4 \equiv 3. \\
    \end{cases}
\end{equation*}
\end{theorem}

\begin{proof} Let $B$ be an arbitrary efficient zero forcing set of $H_m$, $\mathcal F=\{F^{(k)}\}_{k=1}^{\pt(H_m)}$ be a propagating family of forces of $B$ on $H_m$, and $\mathcal C$ be the chain set induced by $\mathcal F$.  We will prove this theorem by proving the subsequent four cases below.
\begin{case}$m \bmod 4 \equiv 0$
\end{case}
Since $m$ is even, by Observation \ref{Pendant_Vertex}, each forcing chain in $\mathcal C$ must start and end at a pendant vertex. Furthermore, by Lemma \ref{Three_Consecutive}, $B$ must contain a consecutive set of three pendant vertices but cannot contain a consecutive set of four pendant vertices (as it would then contain two distinct consecutive sets of three pendant vertices).  Furthermore, since $m \bmod 4 \equiv 0$, it follows that $\abs{B}=\frac{m}{2}$ is even and so $B$ must contain another group of forcing pendants of odd length. By Lemma \ref{Three_Consecutive}, this must be an isolated forcing pendant, $q$. 

During $F^{(1)}$, each pendant vertex in $B$ forces its only neighbor. During $F^{(2)}$, only the center vertex can be forced, since every spoke vertex is adjacent to the center vertex.

Let $u$ be the spoke vertex adjacent to $q$, and note that neither of the pendant vertices which are consecutive with $q$ are members of $B$, so there exist spoke vertices $w,w' \in V(G) \setminus E_{\mathcal F}^{[2]}$ such that $w,w' \in N_G(u)$.  Due to this, neither $u$ nor $v_0$ can force on the third time-step. Since $v_{0}$ and $u$ will need to force in order for their respective forcing chains to end at a pendant vertex, we know that after the third time-step there will be at least two white spoke vertices, say $w_u$ and $w_0$ (either of which may be equal to $w$ or $w'$) such that $u \rightarrow w_u$ and $v_0 \rightarrow w_0$ during $\mathcal F$. 

Furthermore, during the fourth time-step, $u$ might force $w_u$, but $v_{0}$ cannot yet force because it still has at least two white neighbors, specifically $w_u$ and $w_0$. During the fifth time-step, $v_{0}$ may force $w_0$, allowing any remaining white pendant vertices to be forced on the sixth time-step, completing $\mathcal F$.  Thus $\pt(H_{m})\geq6$, for $m \bmod 4 \equiv 0$.

\begin{case}
$m \bmod 4 \equiv 1$
\end{case}

By Observation \ref{Pendant_Vertex}, all but one forcing chain in $\mathcal C$ must contain two pendant vertices with the remaining forcing chain containing one pendant vertex. First suppose $B$ contains a vertex which is not a pendant vertex. By Theorem \ref{term}, we know that $\pt(G,\Term(\mathcal F)) \leq \pt(G, B)$. Furthermore by Observation \ref{Pendant_Vertex}, we know  that every vertex in $\Term(\mathcal F)$ will be a pendant vertex. Due to this, in determining a lower bound on $\pt(G)$, we can assume every vertex in our zero forcing set is a pendant vertex.

Now, suppose without loss of generality that every vertex in $B$ is a pendant vertex. During the first time-step, at most each pendant vertex in $B$ forces its only neighbor.  During the second time-step, at most the center vertex $v_0$ will be forced, since every spoke vertex is adjacent to $v_0$.

Now that $v_0$ is blue, forcing may occur along the cycle of spoke vertices on the third time-step. If the entire cycle of spoke vertices is now blue, then on the fourth time-step, the remaining pendant vertices may be forced, completing $\mathcal F$. Thus $\pt(H_m) \geq 4$, for $m \bmod 4 \equiv 1$.

\begin{case}
$m \bmod 4 \equiv 2$
\end{case}

As in Case 1, since $m$ is even, $\mathcal C$ will have $\frac{m}{2}$ forcing chains, each of which start and end at a pendant vertex. During the first time-step, at most each pendant vertex in $B$ forces its only neighbor. At most, during the second time-step, the center vertex will be forced, since every spoke vertex is adjacent to the center vertex.  Now that the center vertex has been forced, forcing may occur along the cycle of spoke vertices on the third time-step. However, since $m$ is even, every forcing chain in $\mathcal C$ must begin and end at a pendant vertex. Thus there must remain a spoke vertex for the center vertex to force which can not happen until the fourth time-step. Now that the entire cycle of spoke vertices is blue, at the fifth time-step the remaining pendant vertices may be forced, completing $\mathcal F$. Thus $\pt(H_m) \geq 5$, for $m \bmod 4 \equiv 2$.

\begin{case}
$m \bmod 4 \equiv 3$
\end{case}

As in Case 2, we can assume that every vertex in $B$ is a pendant vertex. So suppose, without loss of generality, that $B$ is an efficient zero forcing set such that every vertex in $B$ is a pendant vertex. During the first time-step, at most each pendant vertex in $B$ forces its only neighbor.  Since every spoke vertex is adjacent to the center vertex, during the second time-step, the center vertex $v_0$ will be the only vertex forced.  

By Lemma \ref{Three_Consecutive_Odd}, $B$ must contain at least one set of three consecutive forcing pendants.  Let $b_0$ be the vertex which forces $v_0$ during $\mathcal F$ and note that $b_0$ must be adjacent to the middle vertex of one such set of three consecutive forcing pendants.  Now that $v_0$ has been forced, forcing may occur along the cycle of spoke vertices on the third time-step. However, note that since $v_0$ cannot force until only a single white spoke vertex remains, $v_0$ cannot force until at least time-step $4$. 
 We will now prove that $V(G) \setminus E_{\mathcal F}^{[3]}$ contains spoke vertices.

Note that in this case, $m \bmod 4 \equiv 3$ so we have that $m = 4k + 3$ where $k$ is an integer. Since m is odd, we know that the size of the zero forcing set is $\frac{m+1}{2}$ which is equivalent to $2k + 2$. This shows that the size of the zero forcing set will always be even. In particular, because $m - (2k + 2) = 2k + 1$, we know that we have $2k + 2$ forcing pendants and $2k + 1$ terminal pendants.  We now consider two cases:

\vspace{-0.1in}

\begin{addmargin}[0.2in]{0in}

\begin{subcase}
$B$ contains an isolated forcing pendant $q$.
\end{subcase}

\vspace{0.1in}

Since $b_0$ was used to force $v_0$ during the second time-step, we are left with $2k + 1$ spoke vertices which can potentially be used to force during the third time-step. This means in order for $E_{\mathcal F}^{[3]}$ to contain every spoke vertex, every spoke vertex that is adjacent to a forcing pendant, except $b_0$, must force during time-step $3$.  However, the spoke vertex $u$ which is adjacent to the isolated forcing pendant $q$ will have two white neighbors $w$ and $w'$, and thus cannot force until after either $w$ or $w'$ become blue.  However, by Observation \ref{Pendant_Vertex} only one forcing chain in $\mathcal C$ can contain a single pendant vertex and thus either $v_0$ or $u$ must force during $\mathcal F$. Finally, since the only neighbors of $v_0$ and $u$ which are still white after the second time-step are spoke vertices, $V(G)\setminus E_{\mathcal F}^{[3]}$ contains spoke vertices.

\begin{subcase}
$B$ contains no isolated forcing pendant.
\end{subcase}

Note that in order for two groups of forcing pendants to be distinct, there must be a group of terminal pendants between them. Since the groups of pendant vertices are arranged along a cycle, there must be the same number of groups of forcing pendants as there are groups of terminal pendants. Because the size of the zero forcing set must be even, we know, by Lemma \ref{Three_Consecutive_Odd}, that we need two consecutive sets of three forcing pendants, and thus either a group of forcing pendants of size at least four or two groups of forcing pendants of size three. Since there are an equal number of groups of forcing pendants as there are groups of terminal pendants, there is one less terminal pendant than forcing pendant, and there are no isolated forcing pendants, there must be at least one group of at least three terminal pendants.

Let $w$ be the spoke vertex adjacent to the middle pendant vertex of a group of three terminal pendants. Since, prior to the third time-step, $v_0$ is the only blue neighbor of $w$ and it cannot force on the third time-step, $w \in V(G)\setminus E_{\mathcal F}^{[3]}$.
\end{addmargin}

\vspace{0.1in}

In both sub-cases, $V(G) \setminus E_{\mathcal F}^{[3]}$ must contain spoke vertices. If every spoke vertex is blue after the fourth time-step, then on the fifth time-step, the remaining pendant vertices may be forced, completing $\mathcal F$. Thus $\pt(H_m) \geq 5$, for $m \bmod 4 \equiv 3$.  Finally, by Theorem \ref{upper_helm},
$$ \pt(H_m) = \begin{cases} 
6 & \text{if }m \bmod 4 \equiv 0\\
4 &  \text{if }m \bmod 4 \equiv 1\\
5 &  \text{if }m \bmod 4 \equiv 2 \text{ or } m \bmod 4 \equiv 3.
\end{cases}$$
\end{proof}

We now provide the following theorem that identifies the zero forcing number and propagation time for our generalized helm graph for the case where $s>1$.

\begin{theorem}
	For $s>1$, $\Z\big{(}H(m,s)\big{)}= m(s-1)+ 1 $ and $\pt\big{(}H(m,s)\big{)} = 2.$
\end{theorem}

\begin{figure}[h]
    \centering
    \includegraphics[width=.5\linewidth]{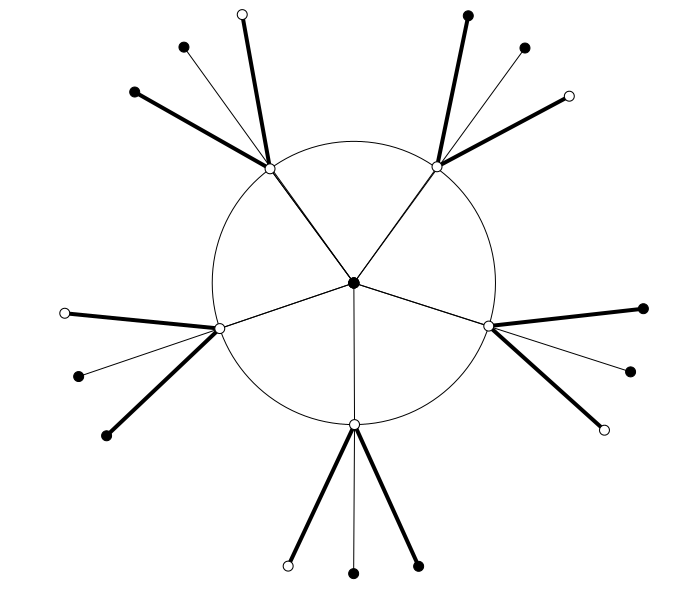}
    \caption{Zero forcing set and chain set of $H(5,3)$ with chain edges in bold}
    \label{GenHelmforce}
\end{figure}

\begin{proof}
	\indent First, we would like to establish $m(s-1)+ 1$ as a lower bound for the zero forcing number. We will do this using two smaller claims. \\
	\indent We claim if $B$ is a zero forcing set of $H(m,s)$, then for each $i$, $\big{\lvert} \{p_{i,j}\}_{j=1}^s \cap B \big{\rvert} \geq s-1$. To show this, suppose by way of contradiction there exists a zero forcing set $B$ of $H(m,s)$ for which there exists $i_0$ such that $\big{\lvert} \{p_{i_0,j}\}_{j=1}^s \cap B \big{\rvert} \leq s-2$. Let $j_1,j_2 \in \{ 1,2, \dots,s\}$ be distinct such that $p_{i_0,j_1}, p_{i_0,j_2} \notin B$. Since $v_{i_0}$ is the only neighbor of $p_{i_0,j_1}$ and $p_{i_0,j_2}$ it is the only vertex which could force them. It can be seen $v_{i_0}$ can do no forcing since $p_{i_0,j_1}, p_{i_0,j_2} \notin B$ and are neighbors of $v_{i_0}$. Therefore, $B$ is not a zero forcing set of $H(m,s)$, and in particular, any zero forcing set of $H(m,s)$ must be such that for each $i$, $\big{\lvert} \{p_{i,j}\}_{j=1}^s \cap Z \big{\rvert} \geq s-1$.\\
	\indent Our second claim states $\Z\big{(}H(m,s)\big{)} \geq m(s-1)  +1$. Let $B$ be a zero forcing set of $H(m,s)$. By our first claim we can assume that up to isomorphism $\{p_{i,j}\}_{i=1,}^m{}_{j=1}^{s-1} \subseteq B$. Then $|B| \geq \big{|} \{p_{i,j}\}_{i=1,}^m{}_{j=1}^{s-1} \big{|} = m(s-1)$. Suppose by way of contradiction $|B| = m(s-1)$. We know in this case that without loss of generality $B = \{p_{i,j}\}_{i=1,}^m{}_{j=1}^{s-1} $. For each $i$, $p_{i,1}$ can force $v_i$. At this point, each $v_i$ has two white neighbors specifically $p_{i,s}$ and $v_0$. Since no other vertex is adjacent to $p_{i,s}$ or $v_0$, neither ever get forced. Hence, $B$ can not be a zero forcing set. So for any zero forcing set $B$, $|B| > m(s-1)$.
	
	Next, we will show $\Z\big{(}H(m,s)\big{)}\leq m(s-1) +1$, establishing an upper bound for the zero forcing number (and thus completing our proof of the zero forcing number) while simultaneously establishing an upper bound on the propagation time by keeping track of the forced vertices. To this end, we want to take a specific zero forcing set of size $m(s-1)+ 1 $ and show that it can force the generalized helm graph. \\
	\indent Let $B = \{v_0\} \cup  \{p_{i,j}\}_{i=1,}^m{}_{j=1}^{s-1}$. We claim $B$ is a zero forcing set and construct a relaxed chronology of forces $\mathcal F$ of $B$ on $H(m,s)$. Note for each $i$, $p_{i,1} \rightarrow v_i$ since $v_i$ is the only neighbor of $p_{i,1}$. So $E_{\mathcal F}^{[1]} = B \cup \{v_i\}_{i=1}^m$. For each $i$, $v_i \rightarrow p_{i,s}$ since $p_{i,s}$ is the only neighbor of $v_i$ not in $E_{\mathcal F}^{[1]}$. Thus $E_{\mathcal F}^{[2]} = E_{\mathcal F}^{[1]} \cup \{p_{i,s}\}_{i=1}^m = V\big{(}H(m,s)\big{)}$. Therefore $H(m,s)$ has been forced with a zero forcing set of size $m(s-1)+ 1$ in two time-steps, thus $\Z\big{(}H(m,s)\big{)} \leq m(s-1)+ 1$ and $\pt\big{(}H(m,s)\big{)}\leq 2.$\\
	\indent Since there is only one vertex in $B$ not in $\{p_{i,j}\}_{i=1,}^m{}_{j=1}^{s-1}$, there is at most one spoke vertex in $B$, say $v_{i_0}$. Then for any $p_{i,s}$, with $i \not = i_0$, not in $B$, there is a distance of at least 2 from $p_{i,s}$ to an element of $B$. Thus $\pt\big{(}H(m,s)\big{)}\geq 2$.  Finally we have that the zero forcing number of $H(m,s)$ must be $ m(s-1)+ 1$ and the propagation time must be $2$.
\end{proof}

\section{Conclusion}

This concludes our study of gear graphs and helm graphs. A natural follow up question is how combining the two classes of graphs by starting with a wheel graph and adding both pendant vertices at spoke vertices as well as intermediate vertices between spoke vertices might affect the zero forcing numbers and propagation times of the resulting class of graphs.  This question forms a basis for research currently being pursued by our team. 

\section*{Acknowledgements}

The research of Kanno Mizozoe was supported by the Trinity College Summer Research Program.  We would also like to thank the Undergrad Research Incubator Program at the University of North Texas for its support.

\bibliographystyle{plain}
\bibliography{helmgear}

\end{document}